\documentclass[12pt,a4paper,article]{article}
\usepackage[utf8]{inputenc}
\usepackage{amsmath}
\usepackage{amsfonts}
\usepackage{amssymb}
\usepackage{amsthm}
\usepackage{dsfont}
\usepackage{authblk}

\theoremstyle{plain}
\newtheorem{theorem}{Theorem}[section]
\newtheorem{lemma}[theorem]{Lemma}

\newtheorem{corollary}[theorem]{Corollary}

\theoremstyle{definition}

\theoremstyle{remark}

\title{Remarks on random walks on graphs and the Floyd boundary}
\author{
        Panagiotis Spanos\thanks{The author acknowledges the support of the Austrian Science Fund (FWF): W1230.
}}
\affil{\footnotesize Institute of Discrete Mathematics, Graz University of Technology, Austria}
\date{}

\begin{document}
\maketitle

\begin{abstract}
We show that for a uniformly irreducible random walk on a graph, with bounded range, there is a Floyd function for which the random walk converges to its corresponding Floyd boundary. Moreover if we add the assumptions, $p^{(n)}(v,w)\leq C \rho^n$, where $\rho < 1$ is the spectral radius, then for any Floyd function $f$ that satisfies $\sum_{n=1}^{\infty}nf(n)<\infty$, the Dirichlet problem with respect to the Floyd boundary is solvable.
\end{abstract}

\section{Introduction}

Different compactifications of classes of graphs,
respectively groups, play important roles in the boundary theory of random walks, typically in connection with
probabilistic and potential theoretic aspects of
boundary convergence. 
The basic compactification of a locally finite, connected graph is its end compactification, see Freudenthal \cite{Frdn}. Any transient random walk with bounded range on such a graph converges to a random end, and under suitable irreducibility assumtions, its Martin compactification, see Martin \cite{Mrtn} and Doob \cite{Doo}, covers the end compactification, see Picardello and
Woess \cite{PWo}. For a class of random walks on hyperbolic graph, respectively on hyperbolic groups (in the sense of Gromov \cite{Grmv}), Ancona \cite{Ancn} showed that the Martin boundary  can be identified with the hyperbolic boundary.
In case of invariance under a transitive group action, very general results are available for convergence to the boundary and the Dirichlet problem at infinity with respect to ends or the hyperbolic compactification. For many references, see Woess \cite{Wo2}. In absence of a group action, Kaimanovich and Woess \cite{KWo} provided geometric criteria beyond bounded range to obtain such results.

In 1980 Floyd \cite{Fld} introduced the Floyd compactification of groups, respectively their Cayley graphs, in order to study limit sets of Kleinian groups, the Floyd compactification is defined for all graphs, respectively groups. Moreover there are groups whose Floyd boundary is bigger than their space of ends. More precisely the Floyd boundary of the fundamental group of a hyperbolic $n$-manifold is the $n$-dimensional sphere.

In this context, it first was Karlsson \cite{K1}, \cite{K2}, \cite{K3} to relate the Floyd boundary with issues concerning random walks. In case of random walks on groups, or more generally, random walks on graphs which are invariant under a transitive group action,
Karlsson has clarified the typical convergence questions. 
Only recently, there has been substantial new work
in the group setting. Gekhtman , Gerasimov , Potyagailo and Yang \cite{GGPY} have shown that under natural assumptions,
the Martin boundary of a random walk on a group covers the
Floyd boundary.

On the other hand, for random walks on graphs that possess
no group invariance, \cite{K3} only contains a few
indications, just concerning simple random walk, under which
conditions one may have convergence to the Floyd boundary. 

The present short note aims at giving more general answers
to those quesitions in the non-group invariant case, along the
line of the methods of \cite{KWo}.
As a matter of fact, it needs a few modifications of those methods
to get analogous results for the Floyd boundary. 

In Section \ref{s2}, convergence to the Floyd boundary is proved for uniformly irreducible random walks with bounded
range, without the typical assumption that the spectral radius is smaller than 1. Instead, the Floyd function is assumed to
depend on the Green function of the random walk.

In Section \ref{s3}, the solution of the Dirichlet problem at infinity and convergence to the boundary are proved
under uniform assumptions concerning irreducibility, first moment and exponential decay of transition
probabilities, which relate the random walk with the underlying geometry.

\section{Convergence to the boundary}\label{s2}

We first describe some preliminaries.
Let $G$ be a locally finite graph and $d$ the graph distance metric. Fix a vertex $e$ and let $f$ be a function $f:\mathbb{N}\to \mathbb{R}_{>0}$ such that
$\lambda f(r)\leq f(r+1) \leq f(r)$ for some constant $\lambda >0$ and every $r>0$. It is required in addition for $f$ to be summable $\sum_{j=0}^{\infty} f(r) < \infty$. The function $f$ is called Floyd function.

We fix a vertex $e$ in the Cayley graph and for an edge $[v,w]$ we denote $d(e,w)=\vert w\vert$ and $\min \lbrace \vert v\vert,\vert w\vert \rbrace =\vert [v,w]\vert$. To create the Floyd boundary, one needs to resize each edge of the graph based to its distance from the fixed vertex $e$ and the Floyd function. So the new length will be $f(\vert [v,w]\vert)$.

Consequently a new length arises, let $\alpha = \lbrace x_i\rbrace_{i=1}^n$ be a path (where $x_i$ and $x_{i+1}$ are neighbours), now the Floyd length of $\alpha$ will be $L_f(\alpha)=\sum_{i=1}^{n-1}f(\vert [x_i,x_{i+1}]\vert)$. Subsequently a new metric arises $d_f(v,w)=\inf\limits_{\alpha} L_f(\alpha)$, where $\alpha$ runs over all paths connecting $v$ and $w$. With this metric $G$ has finite diameter and the completion  $\bar{G}$ of $G$ is a compact metric space and the Floyd boundary is defined as $\partial_f G = \bar{G}\setminus G$.

Let $P$ be the irreducible transition matrix of a transient random walk $Z_n$  on $G$. If there is a $\Gamma\leqslant Aut(G,P)$ that acts transitively on $G$. Karlsson \cite{K2} proved  that $\bar{G}$ satisfies Woess' axioms of projectivity and contractivity, meaning that $\bar{G}$ is a contractive  $\Gamma$-compactification  of $G$ (as defined in \cite{Wo1}). Therefore if the Floyd boundary $\partial_f G$ is infinite and does not have a fixed point by $\Gamma$, then the Dirichlet problem with respect to $P$ and $\bar{G}$ is solvable. 

In order to prove this, Karlsson first proves \cite[Lemma~1]{K2}. Since it will be needed in our proof, we briefly describe it. Let $v,w$ be two vertices  of $G$, set 
$v\wedge w=\tfrac{1}{2}\left( d(v,e)+d(w,e)-d(v,w)\right)$ the Gromov product of $v$ and $w$, let $m$ be the nearest, to $e$, point of the $d$-geodesic $[v,w]$. Then 
\begin{align}
d_f(v,w)\leq d_f(v,m)+d_f(w,m)\leq 2\Big(2\vert m \vert f(\vert m\vert )
+\sum\limits_{i=\vert m\vert}^{\vert m \vert + d(v,w)} f(i)\Big).\label{anis}
\end{align}
Define $\nu(\vert m\vert):= 4\vert m \vert f(\vert m\vert )
+2\sum_{i=\vert m\vert}^{\infty} f(i)$ and observe that $\nu(\vert m \vert) \to 0$ as $\vert m \vert \to \infty$. Moreover if $(v\wedge w)\rightarrow \infty$, then $d_f(v,w)\rightarrow 0$, since $d_f(v,w)\leq\nu(\vert m \vert )$ and
$\vert m \vert = d(e,m) \geq \tfrac{1}{2}\left(d(v,e)+d(w,e)-d(v,w)\right)$. This inequality turns out equally useful in the general case.

With the hypothesis of uniform irreducibility and bounded range, we find a Floyd function for which the random walk converges to the corresponding Floyd boundary. This function depends on the distribution of the random walk. Uniform irreducibility means that there are $\epsilon_0>0$ and $K>0$ such that for every $v,w\in G$ with $v\sim w$ there is $\kappa<K$ such that $p^{(\kappa)}(v,w)\geq \epsilon_0$ and bounded range means that there is an $M>0$ such that $\sup \lbrace d(v,w):v,w\in G, p(v,w)>0\rbrace \leq M$. The Green function $g(v,w)$, for $v,w\in G$, of a random walk $Z_n$ is the expected number of visits to $w$ when $Z_0=v$. If $g(v,w)$ is finite for every $v,w\in G$ then $Z_n$ is called transient. For any set $W\subset G$ we set $g(v,W)=\sum_{w\in W}g(v,w)$.
Let $S_n=\lbrace w\in G : \vert w\vert = n \rbrace$
and $B_n= \lbrace w\in G: \vert w\vert \leq n\rbrace$.

\begin{lemma}\label{lemma}
Let $P$ be the transition matrix of a transient, uniformly irreducible random walk $Z_n$ on $G$ with bounded range. Then $f(n)=\frac{1}{n^3 g(e,B_{n+M})}$ is well defined and satisfies the requirements of a Floyd function.
\end{lemma}
\begin{proof}
The graph is locally finite, thus $S_n$ and $B_n$ are finite sets for every $n\in \mathbb{N}$. The random walk $Z_n$ is transient, therefore $g(e,B_{n+M})$ is finite for every $n\in \mathbb{N}$, hence $f$ is well defined. Since the balls are an increasing sequence of sets, $f$ is decreasing and summable. In order to show that $f$ satisfies the requirements of a Floyd function, it has to be proven that there is a $\lambda>0$ such that $\lambda f(n) \leq f(n+1)$ for every $n\in \mathbb{N}$.

The random walk is uniformly irreducible therefore there is an $\epsilon_0>0$ and a $K<\infty$  such that $v\sim w
\Rightarrow p^{(\kappa)}(v,w) \geq \epsilon_0 \text{ for some } \kappa\leq K$.
This also implies that $\deg (v) \leq \frac{K+1}{\epsilon_0}$. Let $v,w \in G$ with $v\sim w$, then: 
$$\epsilon_0 g(e,w)=\sum_{k=1}^{\infty} \epsilon_0 p^{(k)}(e,w)\leq \sum_{k=1}^{\infty} p^{(\kappa)}(w,v) p^{(k)}(e,w)\leq 
\sum_{k=1}^{\infty} p^{(k)}(e,v)=g(e,v)
$$
Combining the bounded degree of each vertex with the inequality for the Green function of neighbouring vertices we gain:
$$g(e,S_{n+1})=\sum_{w\in S_{n+1}} g(e,w)\leq
\sum_{v_w} \frac{1}{\epsilon_0}g(e,v_w)\leq
\frac{K+1}{\epsilon_0}\frac{1}{\epsilon_0}g(e,S_n)$$
We note that $v_w$ is an arbitrary neighbour of $w$ such that $v_w\in S_n$, moreover the same vertex can appear multiple times since many $w\in S_{n+1}$ could be spanned from the same vertex $v_w$, but at most $\frac{K+1}{\epsilon_0}$, since we have bounded degree. 
We conclude:
\begin{align*}
\left( \frac{n}{n+1}\right)^3 \frac{g(e,B_{n+M})}{g(e,B_{n+M+1})}&=
\left( \frac{n}{n+1}\right)^3 \frac{\sum\limits_{k=0}^{n+M} g(e,S_k)}{\sum\limits_{k=0}^{n+M+1} g(e,S_{k})}\\
&=
\left( \frac{n}{n+1}\right)^3 \frac{\sum\limits_{k=0}^{n+M} g(e,S_k)}{\sum\limits_{k=0}^{n+M} g(e,S_{k})+g(e,S_{n+M+1})}\\
&\geq 
\left( \frac{n}{n+1}\right)^3 \frac{\sum\limits_{k=0}^{n+M} g(e,S_k)}{\sum\limits_{k=0}^{n+M} g(e,S_{k})+\frac{K+1}{\epsilon_{0}^{2}}g(e,S_{n+M})}\\
&\geq \frac{\epsilon_0^2}{8(K+1+\epsilon_0^2)} > 0 \qedhere
\end{align*}
\end{proof}

\begin{theorem}\label{th1}
Let $P$ be the transition matrix of a transient, uniformly irreducible random walk $Z_n$ on $G$ with bounded range. Then there exist a Floyd function $f$ such that, for any starting point, $Z_n$ converges to a point of $\partial_f G$ almost surely.
\end{theorem}
\begin{proof}
The random walk $Z_n$ has bounded range, so there is an $M>0$ such that $d(Z_k,Z_{k+1})<M$ for every $k\in\mathbb{N}$.
We define a new random variable based on the random walk:
$$
X_k:=\max\lbrace \vert Z_k \vert - M, 0\rbrace
$$
The Gromov product is bounded below:
\begin{align*}
Z_k\wedge Z_{k+1}&=\frac{1}{2}(\vert Z_k\vert +\vert Z_{k+1}\vert
-d(Z_k,Z_{k+1}))\\
&\geq \frac{1}{2}(\vert Z_k\vert + \vert Z_{k+1}\vert - M)\\
&\geq \frac{1}{2}(\vert Z_k\vert + \vert Z_{k}\vert - M - M)\\
&= X_k
\end{align*}

We define $f(n)=\dfrac{1}{n^3 g(e,B_{n+M})}$ as the Floyd function, $f$ is well defined by Lemma \ref{lemma}. Observe that $nf(n)$ is a decreasing function, hence so is $\nu(n)$.

Let $m<n$ for $m,n\in \mathbb{N}$.
\begin{align}
d_f(Z_m,Z_n)&\leq \sum\limits_{k=m}^n d_f(Z_k,Z_{k+1})
\leq \sum\limits_{k=m}^n \nu(\lfloor Z_k\wedge Z_{k+1} \rfloor ) \leq \sum\limits_{k=m}^n \nu (X_k)\nonumber \\
&\leq
\sum\limits_{k=m}^n 4 (X_k) f(X_k)
+2\sum\limits_{k=m}^n \sum\limits_{i=X_k}^{X_k + M} f(i)\nonumber \\
&\leq
\sum\limits_{k=m}^n 4X_k f(X_k)
+2\sum\limits_{k=m}^n M f(X_k) \label{ine 3}
\end{align}
The second sum in this inequality becomes small as $m$ grows, since $f$ satisfies the requirements of a Floyd function.

In order to prove convergence to the boundary, we will prove that the sequence is Cauchy almost surely.
We compute
\begin{align*}
\sum\limits_{k=1}^{\infty} \mathbb{E}\left[ X_kf(X_k)\right]
&=
\sum\limits_{k=1}^{\infty} \sum\limits_{n=1}^{\infty} nf(n)\mathbb{P}(X_k =n)\\
&=\sum\limits_{n=1}^{\infty} nf(n)\sum\limits_{k=1}^{\infty} \mathbb{P}(X_k =n)\\
&=\sum\limits_{n=1}^{\infty} n \dfrac{1}{n^3 g(e,B_{n+M})} \sum\limits_{k=1}^{\infty} \mathbb{P}(X_k =n)\\
&=\sum\limits_{n=1}^{\infty} n \dfrac{1}{n^3 g(e,B_{n+M})}
g(Z_0,S_{n+M})\\
&\leq \sum\limits_{n=1}^{\infty} \dfrac{1}{n^2} \dfrac{1}{g(e,B_{n+M})} \epsilon_0^{-\vert Z_0\vert} g(e,S_{n+M})\\
&< \epsilon_0^{-\vert Z_0\vert} \sum\limits_{n=1}^{\infty}\frac{1}{n^2}<\infty
\end{align*}

This leads to $\sum_{k=1}^{\infty} \mathbb{E}\left[ X_kf(X_k)\right]<\infty$, so $\sum_{k=1}^{\infty}  X_kf(X_k) <\infty$ almost surely. Similarly, $\sum_{k=1}^{\infty}f(X_k)<\infty$ almost surely. We conclude that 
$$\sum_{k=1}^{\infty} 4X_k f(X_k)+ 2\sum_{k=1}^{\infty} M f(X_k)<\infty\quad \text{almost surely}
$$
From the inequality (\ref{ine 3}), this implies that $(Z_n)$ is a $d_f$-Cauchy sequence.
\end{proof}

In the Lemma \ref{lemma} we require for the random walk to have bounded range, in order to fix $M$ for the proof of the Theorem \ref{th1}. This assumption could be removed, and $M$ could be taken as an arbitrary natural number.

\section{Dirichlet problem}\label{s3}
In Theorem \ref{th1}, we have achieved convergence to the boundary for a specific Floyd function; essentially this function depends on the rate at which the Gromov product of consecutive points of the random walk tends to infinity. 

In the case of the graph that satisfies the strong isoperimetric inequality, Karlsson \cite{K3} mentions that maybe, following the methods of \cite{KWo}, one could achieve convergence to the boundary and solvability for the Dirichlet problem if the Floyd function $f=o(n^{-3})$ and the Floyd boundary is infinite.

The conditions imposed in \cite{KWo} imply assymptotically linear speed for the Gromov product. This will allow us to prove convergence to the boundary and solvability of the Dirichlet problem for every Floyd function $f$ that satisfies $\sum_{n=1}^{\infty}nf(n)<\infty$ (independently the boundary's cardinality). Therefore for the rest of the statements we require that $f$ satisfies this condition.

For both cases of space of ends and hyperbolic boundary, \cite{KWo} proves lemmas about the asymptotic behavior of the random walk.
Those lemmas are proved for the graph and they do not depend on a specific compactification, therefore they can be used in the case of the Floyd compactification. A slightly more general version can be found in \cite{Wo2} and we shall use those in order to gain a more general result.

A graph $G$ satisfies the strong isoperimetric inequality if there is a $\eta>0$ such that $\vert \partial W \vert \geq \eta \vert W \vert$, for every finite subset of its vertices $W\subseteq G$, where $\partial W = \lbrace w\in W: \text{there exist a } v\sim w \text{ with } v\in G\setminus W\rbrace$.
Let $P$ be the transition matrix of a random walk $Z_n$ on $G$. The spectral radius of $P$ is defined as $\rho (P)= \limsup_{n\to \infty} p^{n} (v,w)^{\frac{1}{n}}$. We say that $(G,P)$ has uniform first moment if $\overline{m}=\sum_{n=1}^{\infty} \phi (n) <\infty$, where $\phi(n)=\sup_{v\in G} \sigma_{v}([n,\infty))$ are the tails of the step length distribution $\sigma_v(n)=\sum\limits_{w:d(v,w)=n} p(v,w)$.

We call $(G,P)$ reversible if there is a measure $m: G \to (0,\infty)$ such that $m(v)p(v,w)=m(w)p(w,v)$ for every $v,w\in G$  and strongly reversible if there is a constant $M>0$ such that $M^{-1}\leq m(v)\leq M$ for every $v\in G$.
For a strongly reversible, uniformly irreducible random walk on a graph, Kaimanovich \cite{Kai} proved that the strong isoperimetric inequality implies spectral radius strictly smaller than 1 and the existance of a constant $C>0$ such that $p^{n}(v,w)\leq C \rho^{n}$ for every $v,w\in G$ and $n\in \mathbb{N}$.
\begin{lemma}\label{lkenou}
Let $\tau(n):= 10 \sum_{i=\lfloor \frac{n}{2}\rfloor+1}^{\infty} f(i)$, then $\tau$ is a decreasing function, $\nu(n)\leq \tau(n)$ for every $n\in \mathbb{N}$ and $\sum_{i=1}^{\infty} \tau (i) < \infty$.
\end{lemma}
\begin{proof}
Obviously $\tau$ is decreasing, since $f$ is also decreasing we gain, 
$$nf(n)\leq 2\left(f\left(\left\lfloor\frac{n}{2}\right\rfloor+1\right)+f\left(\left\lfloor\frac{n}{2}\right\rfloor+2\right)\cdots f\left(2\left\lfloor \frac{n}{2}\right\rfloor+1\right)\right)< 2 \sum_{i=\lfloor \frac{n}{2}\rfloor+1}^{\infty} f(i).$$
This yields
$$\nu(n)=4nf(n)+2\sum_{i=n}^{\infty}f(i)<8\sum_{i=\lfloor \frac{n}{2}\rfloor+1}^{\infty}f(i)+2\sum_{i=\lfloor \frac{n}{2}\rfloor+1}^{\infty}f(i)= \tau (n).$$ 
To complete the proof we make the following calculations $$\sum_{n=1}^{\infty} \tau(n)=
\sum_{n=1}^{\infty}10\sum_{i=\lfloor \frac{n}{2}\rfloor+1}^{\infty} f(i) \leq 10 \sum_{n=1}^{\infty} 2 n f(n) < \infty. 
$$
\end{proof}
\begin{lemma}\label{2}
Let $P$ be the transition matrix of a uniformly irreducible random walk $Z_n$ on $G$ with uniform first moment and $\rho(P)<1$. Then $Z_n$ converges to a random point $\xi$ of $\partial_f G$ almost surely.
\end{lemma}
\begin{proof}
Since $Z_n$ satisfies those conditions,  we gain the following asymptotic\linebreak
properties  
$\lim\limits_{n\to\infty}\frac{1}{n}d(Z_n,Z_{n+1})=0$ and 
$0<\underline{m}\leq \liminf\limits_{n\to\infty}\frac{1}{n}\vert Z_n\vert$ almost surely (from \cite[Lemma 8.8 and Corollary 8.9]{Wo2}).
Combining those two, we gain that
$\liminf\limits_{n\to \infty} \tfrac{1}{n} (Z_n\wedge Z_{n+1}) \geq \underline{m}$ almost surely. 

So there are $c_0>0$ and $n_0\in \mathbb{N}$ such that $Z_n\wedge Z_{n+1}> c_0 n$ for every $n\geq n_0$ almost surely. We set $c_1=min\lbrace c_0, 1\rbrace$. From inequality (\ref{anis}) and Lemma \ref{lkenou} we conclude that:
\begin{align*}
\sum\limits_{k=1}^{\infty} d_f(Z_k,Z_{k+1})&\leq
\sum\limits_{k=1}^{\infty}\tau ( \lfloor Z_k\wedge Z_{k+1}\rfloor )\\
&\leq \sum\limits_{k=1}^{n_0}\tau ( \lfloor Z_k\wedge Z_{k+1}\rfloor ) + \sum\limits_{k=n_0}^{\infty}\tau ( \lfloor c_0 n \rfloor )\\
&\leq \sum\limits_{k=1}^{n_0}\tau ( \lfloor Z_k\wedge Z_{k+1}\rfloor ) + \sum\limits_{k=1}^{\infty}c_1^{-1}\tau (n) < \infty
\end{align*}
This yields that $Z_n\to \xi$ almost surely for a random $\xi \in \partial_f G$.
\end{proof}

The following Lemma has the same core as \cite[Corollary 2]{KWo}, which was proved and used in order to show the solvability of the Dirichlet problem in both cases of ends of graphs and hyperbolic boundary. Once more we provide a slightly more general Lemma that will follow \cite[Lemma 21.17]{Wo2}. We will make some adjustments in order to work for the case of Floyd compactification. 

\begin{lemma}\label{3}
Let $P$ be the transition matrix of a uniformly irreducible random walk $Z_n$ on $G$ with uniform first moment, $\rho(P)<1$ and $p^{(n)}(v,w)\leq C\rho^{n}$ for a constant $C>0$ and every $v,w\in G$, $n\in\mathbb{N}$. For $\alpha,\epsilon >0$ and $v\in G$ we define $\mathcal{A}_v=\mathcal{A}_v(\alpha,\epsilon)$ as the event that:
\begin{enumerate}
\item $Z_0=v$
\item $d(Z_0,Z_n)\leq (\overline{m}+\epsilon)n$, $n\geq \alpha
\vert v\vert$ 
\item $d(Z_n,Z_{n+1})\leq \epsilon n$, $n>\alpha\vert v\vert$
\item $\vert Z_n\vert \geq (\underline{m}-\epsilon)n$, $n\geq \alpha \vert v \vert$
\item $\vert Z_n\vert > \epsilon \vert v\vert$, $n\geq 0$
\end{enumerate}
Then there is $\epsilon_0>0$ such that for all $\epsilon\leq \epsilon_0$ and $\alpha>0$
$\lim\limits_{\vert v\vert \to \infty}\mathbb{P}_v(\mathcal{A}_v)=1$.
\end{lemma}
\begin{proof}
Let $\mathcal{A}_{v,i}$ be the event that the $i$-th property is satisfied, for $i=2,3,4,5$. Then from \cite[Lemma 21.17]{Wo2}, it is known that $\lim\limits_{\vert v\vert \to \infty} \mathbb{P}_{v}(\mathcal{A}_{v,i})=1$ for $i=3,5$.
As mentioned before uniform first moment yields that $$\limsup\limits_{n\to\infty}\frac{1}{n}\sup\limits_{k\leq n} d(Z_0,Z_k)\leq \overline{m}\text{ and } \liminf\limits_{n\to \infty}\frac{1}{n}\vert Z_n\vert \geq \underline{m} \quad \text{almost surely}$$
(see \cite[Lemma 8.8 and Corollary 8.9]{Wo2}) Therefore $\limsup\limits_{n\to\infty} \frac{1}{n}d(Z_0,Z_n)\leq \overline{m}$ almost surely  so $\lim\limits_{n\to\infty}\mathbb{P}(\mathcal{A}_{v,i})=1$ for $i=2,4$.
\end{proof}

The main idea in \cite{KWo} is to bound the probability that the random walk converges to a neighbourhood of $\xi$ by the probability that an event $\mathcal{A}_v$ occurs. We follow this plan.

\begin{theorem}
Let $P$ be the uniformly irreducible transition matrix of a random walk $Z_n$ on $G$ with uniform first moment, $\rho(P)<1$ and $p^{(n)}(v,w)\leq C\rho^{n}$ for a constant $C>0$ and every $v,w\in G$, $n\in\mathbb{N}$. Then the Dirichlet problem with respect to the Floyd compactifiation is solvable.
\end{theorem}
\begin{proof}
Let $\mu_v$ be the Borel measure in $\partial_f G$ for the $\mathbb{P}_v$-distribution of $\lim\limits_{n\to\infty}Z_n$, meaning $\mu_v(B)=\mathbb{P}[\lim\limits_{n\to\infty}Z_n\in B\vert Z_0=v]$.
The Dirichlet problem is solvable if and only if the random walk converges to the boundary and the measures $\mu_v$ converges weakly to the Dirac measure $\delta_\xi$ as $v\to\xi$, for every $\xi \in \partial_f G$ (see Theorem (20.3) \cite{Wo2}). 
From Lemma \ref{2} we already know that the random walk converges to a point on the boundary almost surely. 

Let $\xi\in\partial_f G$ and $B(\xi,r)$ be the open ball in $\bar{G}$, let $v\in B(\xi,\frac{r}{2})\cap G$ and we take $\mathcal{A}_v=\mathcal{A}_v(\epsilon,\alpha)$ as in the previous Lemma, where $\alpha=\frac{1}{2\overline{m}}$, $\epsilon< \frac{1}{3}\underline{m}$. If $\vert v \vert > \frac{\epsilon + \overline{m}}{\epsilon}$ and the event $\mathcal{A}_v$ occurs then:\\
Set $k=\lceil\alpha\vert v\vert\rceil=\lceil\frac{1}{2\overline{m}}\vert v\vert\rceil$, then
\begin{align*}
2 (Z_0\wedge Z_k )&=\vert Z_0\vert+\vert Z_k \vert - d(Z_0,Z_k) =\vert v \vert + \vert Z_k\vert - d(v,Z_k)\\
&\geq \vert v\vert + \epsilon \vert v \vert  - (k+1)(\overline{m}+\epsilon)
\geq \vert v\vert (1+\epsilon) - (\frac{1}{2\overline{m}}\vert v\vert -1)(\overline{m}+\epsilon)\\
&\geq \vert v\vert (1+\epsilon-\frac{1}{2}-\frac{\epsilon}{2\overline{m}}-\frac{\overline{m}+\epsilon}{\vert v \vert})\geq \vert v \vert (1 - \frac{1}{2}-\frac{1}{6}+
\frac{\epsilon\vert v \vert - \epsilon - \overline{m}}{\vert v \vert})\\
&\geq \frac{\vert v \vert}{3}
\end{align*}
If $n\geq \alpha \vert v\vert$, then
\begin{align*}
2 (Z_n\wedge Z_{n+1}) &=\vert Z_n \vert + \vert Z_{n+1}\vert - d(Z_n,Z_{n+1})\\
&\geq (\underline{m}-\epsilon)n+(\underline{m}-\epsilon)(n+1)-\epsilon n\\
&\geq (2\underline{m}-3\epsilon)n+\underline{m}-\epsilon\\
&\geq (2\underline{m}-3\epsilon)n
\end{align*}
Consequently for $n\geq \alpha\vert v\vert$
\begin{align*}
d_f(Z_0,Z_n)&\leq d_f(Z_0,Z_k)+d_f(Z_k,Z_n)\\
&\leq \tau ( \lfloor v\wedge Z_k \rfloor )+\sum\limits_{i=k}^n \tau ( \lfloor Z_i\wedge Z_{i+1} \rfloor )\\
&\leq \tau \left(\left\lfloor\frac{\vert v\vert}{6}\right\rfloor\right)+\sum\limits_{i=k}^n \tau \left(\lfloor (\underline{m}-\frac{3\epsilon}{2})i\rfloor\right)
\end{align*}
Now we set $c=\min\lbrace \underline{m}-\frac{3\epsilon}{2}, 1\rbrace$ and the inequality becomes: 
\begin{align*}
d_f(Z_0,Z_n)&\leq \tau \left(\left\lfloor\frac{\vert v\vert}{6}\right\rfloor\right)+\sum_{i=k}^n\tau\left( \lfloor c i\rfloor \right)\\
&
\leq \tau \left( \left\lfloor \frac{\vert v\vert}{6} \right\rfloor \right)
+\sum\limits_{i=\lfloor c k\rfloor}^{\infty} c^{-1} \tau(i)
\end{align*}

Since we have from Lemma \ref{lkenou} for $\tau$ that $\sum_{i=1}^{\infty}\tau(i)<\infty$, there is a $\vert v\vert>\frac{\overline{m}+\epsilon}{\epsilon}$ large enough such that $d_f(Z_0,Z_n)<\frac{r}{2}$ for every $n>\alpha\vert v \vert$. This means that $Z_n\in B(\xi,r)$ for all large $n$. This can be expressed as
$\mu_{v}(B(\xi,r)\cap\partial_f G)\geq \mathbb{P}_{v}(\mathcal{A}_v)$, but from Lemma \ref{3} $\lim\limits_{\vert v\vert \to\infty}\mathbb{P}_{v}(\mathcal{A}_v)=1$. This is true for every $r>0$, therefore $\lim\limits_{v\to\xi}\mu_{v}=\delta_{\xi}$.
\end{proof}

\begin{corollary}
Let $G$ be a graph that satisfies the strong isoperimetric inequality and $P$ be the uniformly irreducible transition matrix of a strongly reversible random walk $Z_n$ that has uniform first moment. Then the random walk $Z_n$ converges to the boundary $\partial_f G$ and the Dirichlet problem with respect to the Floyd compacftifiation is solvable.
\end{corollary}

Therefore there exist an analogue of probabilistic propositions for the Floyd boundaries as there is for hyperbolic boundaries and the space of ends. As already mentioned in the introduction, for a random walk on a locally finite hyperbolic graph there is a relation between the hyperbolic boundary and the Martin boundary. Also, for a random walk on an arbitrary graph, the Martin boundary covers the space of ends. The recent work of Gekhtman, Gerasimov, Potyagailo and Yang sparks the question of whether the Martin boundary of a random walk covers the Floyd boundary (for suitable Floyd functions) on a general graph under geometric adaptedness conditions in the place of group invariance.
\bibliography{vivlio}

\begin{thebibliography}{10}

\bibitem{Ancn}
Alano Ancona.
\newblock Positive harmonic functions and hyperbolicity.
\newblock In Josef Kr{\'a}l, Jaroslav Luke{\v{s}}, Ivan Netuka, and
  Ji{\v{r}}{\'i} Vesel{\'y}, editors, {\em Potential Theory Surveys and
  Problems}, pages 1--23, Berlin, Heidelberg, 1988. Springer Berlin Heidelberg.

\bibitem{Doo}
J.~L. Doob.
\newblock Discrete potential theory and boundaries.
\newblock {\em Journal of Mathematics and Mechanics}, 8(3):433--458, 1959.

\bibitem{Fld}
William~J. Floyd.
\newblock Group completions and limit sets of kleinian groups.
\newblock {\em Inventiones mathematicae}, 57(3):205--218, Oct 1980.

\bibitem{Frdn}
Hans Freudenthal.
\newblock {\"U}ber die enden diskreter r{\"a}ume und gruppen.
\newblock {\em Commentarii Mathematici Helvetici}, 17(1):1--38, Dec 1944.

\bibitem{GGPY}
Ilya Gekhtman, Victor Gerasimov, Leonid Potyagailo, and Wenyuan Yang.
\newblock Martin boundary covers floyd boundary.
\newblock {\em Inventiones Math.}, to appear.

\bibitem{Grmv}
M.~Gromov.
\newblock {\em Hyperbolic Groups}, pages 75--263.
\newblock Springer New York, New York, NY, 1987.

\bibitem{Kai}
Vadim~A. Kaimanovich.
\newblock Dirichlet norms, capacities and generalized isoperimetric
  inequalities for markov operators.
\newblock {\em Potential Analysis}, 1(1):61--82, Mar 1992.

\bibitem{KWo}
Vadim~A. Kaimanovich and Wolfgang Woess.
\newblock The dirichlet problem at infinity for random walks on graphs with a
  strong isoperimetric inequality.
\newblock {\em Probability Theory and Related Fields}, 91(3):445--466, Sep
  1992.

\bibitem{K1}
Anders Karlsson.
\newblock Boundaries and random walks on finitely generated infinite groups.
\newblock {\em Arkiv f{\"o}r Matematik}, 41(2):295--306, Oct 2003.

\bibitem{K2}
Anders Karlsson.
\newblock Free subgroups of groups with nontrivial floyd boundary.
\newblock {\em Communications in Algebra - COMMUN ALGEBRA}, 31, Jan 2003.

\bibitem{K3}
Anders Karlsson.
\newblock {Some remarks concerning harmonic functions on homogeneous graphs}.
\newblock volume DMTCS Proceedings vol. AC, Discrete Random Walks (DRW'03),
  pages 137--144, 2003.

\bibitem{Mrtn}
Robert~S. Martin.
\newblock Minimal positive harmonic functions.
\newblock {\em Transactions of the American Mathematical Society},
  49(1):137--172, 1941.

\bibitem{PWo}
Massimo~A. Picardello and Wolfgang Woess.
\newblock Harmonic functions and ends of graphs.
\newblock {\em Proceedings of the Edinburgh Mathematical Society},
  31(3):457–461, 1988.

\bibitem{Wo1}
Wolfgang Woess.
\newblock Fixed sets and free subgroups of groups acting on metric spaces.
\newblock {\em Mathematische Zeitschrift}, 214(1):425--439, Sep 1993.

\bibitem{Wo2}
Wolfgang Woess.
\newblock {\em Random Walks on Infinite Graphs and Groups}.
\newblock Cambridge Tracts in Mathematics. Cambridge University Press, 2000.

\end{thebibliography}
\bibliographystyle{plain}
\end{document}